\documentclass[12pt,reqno]{amsart}
\newcommand\theoremnumbering{subsection}   
\usepackage{tikz}
\usepackage{tikz-cd}

\usepackage{fullpage,setspace}
\usepackage{mathrsfs} 
\newcommand\choosefont[1]{\usepackage{#1}}
\usepackage[bottom]{footmisc}  
\usepackage{enumerate}   
\usepackage{url,cite,fancyheadings}
\usepackage{enumerate}
\usepackage{asymptote}
\usepackage{amsmath,amssymb,amscd,mathrsfs,latexsym}
\usepackage{mathdots}

\usepackage{ifthen}     

\newcommand\pubpri[2]{%
\ifthenelse{\equal{\version}{public}}%
{{#1}}%
{\ifthenelse{\equal{\finalized}{no}}{\marginpar{\scshape\small Pubpri Alert}{#2}}{#2}}{}}
\newcommand\pubprinoalert[2]{%
\ifthenelse{\equal{\version}{public}}%
{{#1}}%
{#2}}

\newcommand\ignore[1]{}

\usepackage{color}
\providecommand\wantcolor{yes}   %
\definecolor{backgroundyellow}{cmyk}{.2,.1,.8,.2}
\definecolor{backgroundblue}{rgb}{0,0,1}
\definecolor{backgroundred}{rgb}{1,0,0}
\definecolor{backgroundmagenta}{cmyk}{0,1,0,0}
\ifthenelse{\equal{\wantcolor}{no}}{\renewcommand\color[1]{}}{}

\newcommand\mysubsubsection[1]{%
		\subsubsection{\sffamily\upshape\mdseries #1}}

\newcommand\mysss{\mysubsubsection}

\usepackage{chngcntr}
\counterwithin{figure}{section}

\usepackage{url,hyperref}  
\usepackage{hyperref}
\hypersetup{colorlinks=true,
linkcolor=blue,
filecolor=magenta,
urlcolor=cyan,}
\providecommand{\theoremnumbering}{document}

\ifthenelse{\equal{\theoremnumbering}{section}}{\newtheorem{annotation}{Annotation}[section]}%
{\ifthenelse{\equal{\theoremnumbering}{subsection}}{\newtheorem{annotation}{Annotation}[subsection]}{\newtheorem{annotation}{Annotation}}}
\newtheorem{theorem}[annotation]{
		Theorem}
\newtheorem{lemma}[annotation]{
		Lemma}
\newtheorem{definition}[annotation]{
		Definition}
\newtheorem{corollary}[annotation]{
		Corollary}
\newtheorem{proposition}[annotation]{
		Proposition}

\newtheorem{example}[annotation]{
		Example}
\newcommand\bexample{\begin{example}\begin{rm}}
\newcommand\eexample{\end{rm}\hfill$\Box$\end{example}}
\newtheorem{examplenobox}[annotation]{
		Example}
\newcommand\bexamplenobox{\begin{examplenobox}\begin{rm}}
\newcommand\eexamplenobox{\end{rm}\end{examplenobox}}
\newtheorem{exercise}[annotation]{
		Exercise}
\newcommand\bexercise{\begin{exercise}\begin{rm}}
\newcommand\eexercise{\end{rm}\end{exercise}}
\newtheorem{notation}[annotation]{
		Notation}
\newcommand\bnotation{\begin{notation}\begin{rm}}
\newcommand\enotation{\end{rm}\end{notation}}

\newtheorem{remark}[annotation]{
		Remark}
\newcommand\bremark{\begin{remark}
\begin{upshape}}
\newcommand\eremark{\end{upshape}
\end{remark}}
\newenvironment{remark*}{%
\par\noindent{\scshape 
  Remark: }\begin{rm}}{\hfill\end{rm}\newline} 
\newcommand\bremarkstar{\begin{remark*}}
\newcommand\eremarkstar{\end{remark*}}
\newcommand\bdefn{\begin{definition}
\begin{upshape}}
\newcommand\edefn{\end{upshape}
\end{definition}}
\newtheorem{caveat}[annotation]{
		Caveat}
\newcommand\bcaveat{\begin{caveat}
\begin{upshape}}
\newcommand\ecaveat{\end{upshape}
\end{caveat}}
\newenvironment{caveatstar}{
\par\noindent{\scshape\bfseries
  Caveat: }\begin{rm}}{\end{rm}\newline} 
\newcommand\bcaveatstar{\begin{caveatstar}}
\newcommand\ecaveatstar{\end{caveatstar}}
\newenvironment{myproof}{%
\par\noindent{\scshape 
  Proof: }\begin{rm}}{\hfill$\Box$\end{rm}\newline} 
\newcommand\bmyproof{\begin{myproof}}
\newcommand\emyproof{\end{myproof}}
\newenvironment{myproofnobox}{%
\par\noindent{\scshape Proof: }\begin{rm}}{\end{rm}\hfill\newline}
\newcommand\bmyproofnobox{\begin{myproofnobox}}
\newcommand\emyproofnobox{\end{myproofnobox}}
\newenvironment{solution}{%
\par\noindent{\scshape Solution: }\begin{rm}}{\hfill$\Box$\end{rm}\newline}
\newenvironment{solutionnobox}{%
\par\noindent{\scshape Solution: }\begin{rm}}{\end{rm}}
\newcommand\bsolution{\begin{solution}\begin{rm}}
\newcommand\esolution{\end{rm}\end{solution}}
\newcommand\bsolutionnobox{\begin{solutionnobox}\begin{rm}}
\newcommand\esolutionnobox{\end{rm}\end{solutionnobox}}
\newcommand\bthm{\begin{theorem}}
\newcommand\ethm{\end{theorem}}
\newcommand\bcor{\begin{corollary}}
\newcommand\ecor{\end{corollary}}
\newcommand\blemma{\begin{lemma}}
\newcommand\elemma{\end{lemma}}
\newcommand\bprop{\begin{proposition}}
\newcommand\eprop{\end{proposition}}
\newcommand\beqn{\begin{equation}}
\newcommand\eeqn{\end{equation}}
\newcommand\beqnstar{\begin{equation*}}
\newcommand\eeqnstar{\end{equation*}}
\usepackage{amsthm}

\newcommand\mtitle[1]%
{\marginpar{\sffamily\raggedright\upshape\small #1}}

\providecommand\finalized{yes}
\ifthenelse{\equal{\finalized}{yes}}{}{\marginparwidth=2cm}
\ifthenelse{\equal{\finalized}{yes}}%
{}%
{}
\ifthenelse{\equal{\finalized}{yes}}%
{\newcommand\checked[1]{}}%
{\newcommand\checked[1]{\marginpar{[{\ttfamily\upshape\tiny CHECKED: #1}]}}}
\ifthenelse{\equal{\finalized}{yes}}%
{\newcommand\spellchecked[1]{}}%
{\newcommand\spellchecked[1]{\marginpar{[{\ttfamily\upshape\tiny SPELLCHECKED: #1}]}}}

\providecommand\version{public}   
\ifthenelse{\equal{\version}{public}}%
{\newcommand\mcomment[1]{}}%
{\newcommand\mcomment[1]{\marginpar{{\raggedright\sffamily\upshape\small
\begin{spacing}{0.75} #1\end{spacing}}}}}
\ifthenelse{\equal{\version}{public}}%
{\newcommand\fcomment[1]{}}%
{\newcommand\fcomment[1]{\footnote{#1}}}
\ifthenelse{\equal{\version}{public}}%
{\newcommand\comment[1]{}}%
{\newcommand\comment[1]{{\small #1}}}

\choosefont{rsfso}    

\usepackage{amsthm}
\usepackage{empheq}
\usepackage{accents}
\usepackage[all, poly]{xy}
\usepackage[utf8]{inputenc}
\usepackage{mathtools}
\usepackage{soul}
\usepackage{tikz}
\usetikzlibrary{matrix,arrows,decorations.pathmorphing}
\usepackage[left=2.5cm,right=2.5cm,top=2cm,bottom=2cm,includeheadfoot]{geometry}

\theoremstyle{plain}

\theoremstyle{remark}

\makeatletter
\newcommand{\ostar}{\mathbin{\mathpalette\make@circled\star}}
\newcommand{\make@circled}[2]{%
  \ooalign{$\m@th#1\smallbigcirc{#1}$\cr\hidewidth$\m@th#1#2$\hidewidth\cr}%
}
\newcommand{\smallbigcirc}[1]{%
  \vcenter{\hbox{\scalebox{0.77778}{$\m@th#1\bigcirc$}}}%
}
\makeatother

\newtheorem{procedure}{Procedure}

\newcommand{\integers}{\mathbb{Z}}

\newcommand{\lie}{\mathfrak}
\newcommand{\wnot}{w_0}

\renewcommand{\iff}{\Leftrightarrow}

\newcommand{\rootthree}{1.7320508}

\DeclareMathOperator{\demcrys}{Dem}

\DeclareMathOperator{\U}{\mathfrak{U}}

\DeclareMathOperator{\GT}{GT}
\DeclareMathOperator{\GTz}{GT_\integers}
\DeclareMathOperator{\kogan}{K}
\DeclareMathOperator{\koganz}{K_\integers}

\DeclareMathOperator{\len}{len}

\newcommand\tensor\otimes

\newcommand\lamdba\lambda
\newcommand\kostant{K}
\newcommand\klwm{\kostant(\lambda,w,\mu)}

\newcommand\klmw\klwm

\newcommand\biject\Psi
\newcommand\cleq\prec
\newcommand\cgeq\succ

\providecommand\tensor\otimes

\newcounter{cnt}
 \makeatletter
\def\mydggeometry{\makeatletter\dg@YGRID=1\dg@XGRID=20\unitlength=0.003pt\makeatother}
\makeatother \theoremstyle{remark}

\numberwithin{equation}{section}
\makeatletter
\def\section{\def\@secnumfont{\mdseries}\@startsection{section}{1}%
  \z@{.7\linespacing\@plus\linespacing}{.5\linespacing}%
  {\normalfont\scshape\centering}}
\def\subsection{\def\@secnumfont{\bfseries}\@startsection{subsection}{2}%
  {\parindent}{.5\linespacing\@plus.7\linespacing}{-.5em}%
  {\normalfont\bfseries}}
\makeatother

\begin{document}

\title[Gelfand-Tsetlin model of Kostant-Kumar modules]{Gelfand-Tsetlin crystals of  Kostant-Kumar modules}

\author[]{Mrigendra Singh Kushwaha}
\address{Department of Mathematics, Faculty of Mathematical Science, University of Delhi, New Delhi - 110007}
\email{mrigendra154@gmail.com, mskushwaha@maths.du.ac.in}

\date{}

\begin{abstract}
We give Gelfand-Tsetlin crystals for the Kostant-Kumar modules for the finite simple Lie algebra of type $A$. Kostant-Kumar modules are cyclic submodules of the tensor product of two irreducible highest weight modules of a symmetrizable Kac-Moody Lie algebras. In this case (type $A$), we also provide a polytopal model for Kostant-Kumar modules in terms of \emph{BiKogan faces}.
\end{abstract}
 
\maketitle 

 \section{Introduction}
\subsection{} Let $\lie g$ be a symmetrizable Kac-Moody Lie algebra, $\lie h$ be a Cartan subalgebra, $\lie b$ be the Borel subalgebra containing $\lie h$, $\lie b^-$ be the opposite Borel subalgebra to $\lie{b}$, $P$ be the weight lattice, $P^+$ be the set of dominant integral weights and $W$ be the Weyl group of $\lie g$. Let $V(\lambda)$ denote the irreducible highest weight module of $\lie g$ for $\lambda \in P^+$. Given $\lambda,\mu \in P^+$ and $w \in W$, the Kostant-Kumar (KK) module $K(\lambda,w,\mu)$ is a cyclic submodule of $V(\lambda)\otimes V(\mu)$ generated by the vector $v_\lambda \otimes v_{w\mu}$ (see \cite{KRV}), that is
\begin{equation}
    K(\lambda,w,\mu):= \lie{Ug}(v_\lambda \otimes v_{w\mu}),
\end{equation}
where $v_\lambda$ is the highest weight vector of $V(\lambda)$, $v_{w\mu}$ be a non zero weight vector of weight $w\mu$ in $V(\mu)$ and $\lie{Ug}$ be the universal enveloping algebra of $\lie g$. The following results are well-known (see \cite{KRV}):
\begin{enumerate}
    \item $K(\lambda, 1,\mu) \cong V(\lambda + \mu) $ where $1$ denote the identity element of $W$,
    \medskip\item $K(\lambda,w_0,\mu) = V(\lambda)\otimes V(\mu) $ where $w_0$ is the longest element of $W$,
    \medskip\item\label{kk:it3} $K(\lambda,w,\mu)\subseteq K(\lambda,w',\mu)$ if $w \leq w'$ in the Bruhat order in $W$.
\end{enumerate}
The item \eqref{kk:it3} shows that KK submodules form a filtration of $V(\lambda)\otimes V(\mu)$ with respect to the Bruhat order in $W$. Further results on KK-modules can be found in the following references \cite{kumar:refineprv, Kumar, mathieu, KRV, KRV_Sat}.

In our paper \cite{KRV}, we gave a path model for the KK-modules in terms of Lakshmibai-Seshadri (LS) paths provided $\lie g$ is either symmetric or of finite type, which is a generalization of the Littelmann's path model for the tensor product of two irreducible highest weight modules of $\lie g$ \cite{litt:inv, LittelmannAnnals}.

In this paper, we consider $\lie g = \lie{sl}_n(\mathbb{C})$ the set of $n\times n$ complex matrices with trace zero and provide a polytopal model for the KK-modules in terms of pairs of Gelfand-Tsetlin patterns (see theorem \ref{thm:gt-mod}) which is a very simple description. It is a generalization of well known polytopal model in terms of pairs of Gelfand-Tsetlin patterns for the tensor product of two irreducible highest weight modules of $\lie g$. The crucial ingredient for the above polytopal model of KK-modules is an association of permutation $\lie p(P,Q)$ to a given pair of GT patterns $(P,Q)$. The description of $\lie{p}(P,Q)$ is visually simple and a procedure for it is stated in section \ref{subsec:asso-perm}.
 
\subsection{} We recall that for $w \in W$ and $\mu \in P^+$, the {\em Demazure module} $V_w(\mu)$ is the cyclic $\U\lie b$-submodule and the {\em opposite Demazure module} $V^w(\mu)$ is the cyclic $\U\lie b^-$-submodule of $V(\mu)$ generated by a nonzero vector $v_{w\mu}$ of weight $w\mu$, where $\U\lie b$ and $\U\lie b^-$ are the universal enveloping algebra of $\lie{b}$ and $\lie{b}^-$ respectively. 

In section \ref{s:bikogan}, we define \emph{BiKogan faces} of a product of two Gelfand-Tsetlin polytopes, in the same spirit of \emph{Kogan faces} and \emph{dual Kogan faces} (see \cite{fujita}, \cite{kogan}) of a Gelfand-Tsetlin polytope. A certain union of \emph{Kogan faces} (resp. \emph{dual Kogan faces}) gives a polytopal model for \emph{Demazure crystals} (resp. \emph{opposite Demazure crystals}). In the same spirit, our polytopal model for the KK-modules is described as a certain union of \emph{BiKogan faces} (see Theorem \ref{thm:Bikogan}). 

\subsection{Acknowledgement} I thank Sankaran Viswanath for his insightful suggestions and comments on the draft of the paper. I also extend my thanks to K. N. Raghavan for many useful discussions.

\section{Gelfand-Tsetlin pattern}\label{sec:gtpattern}

\subsection{} Let $\lie g = \lie{sl}_n(\mathbb{C})$. In this case, the Weyl group $W$ can be identified with the permutation group $S_n$ of the set $\{1,2,\ldots,n\}$ and the set of dominant integral weights $P^+$ with the set of integer partitions with $n$-parts such that the $n^{th}$-part is zero. A sequence $\mu = (\mu_1,\mu_2,\ldots,\mu_n) \in \mathbb{Z}^n_{\geq 0}$ is called an integer partition if $\mu_i \geq \mu_{i+1}$ for all $i=1,2,\ldots,n-1$.

There are many combinatorial models for the irreducible representation $V(\mu)$ indexed by $\mu \in P^+$. 
The set of integral Gelfand-Tsetlin (GT) patterns of shape $\mu$ is one of the combinatorial models of $V(\mu)$, introduced by Gelfand and Tsetlin \cite{GT}. 

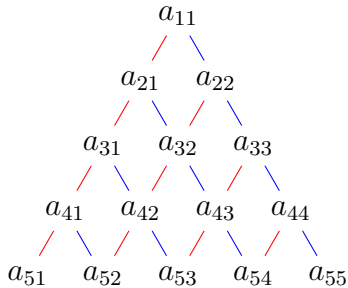
\begin{figure}
\begin{center}
\begin{tikzpicture}[x={(1cm*0.5,-\rootthree cm*0.5)},y={(1cm*0.5,\rootthree cm*0.5)}]
\foreach\i in{0,...,4}
\foreach\j in{\i,...,4}{
  \pgfmathtruncatemacro{\k}{5 - \j + \i};
  \pgfmathtruncatemacro{\l}{\i + 1};
  \draw(\i,\j)node(a\i\j){$a_{\k\l}$};
}
  \foreach\i/\ii in{0/-1,1/0,2/1,3/2,4/3}
 \foreach\j/\jj in{0/-1,1/0,2/1,3/2,4/3}{
  \ifnum\i<\j \draw[color=red](a\i\jj)--(a\i\j); \fi
  \ifnum\i>\j\else\ifnum\i>0 \draw[color=blue](a\ii\j)--(a\i\j);\fi\fi
 }
\end{tikzpicture}\qquad

\caption{Gelfand-Tsetlin array for $n=5$.}
\label{fig:gt-array-aij-specified}
\end{center}
\end{figure}

\subsection{} A GT-pattern of size $n$ is a triangular array $A=(a_{ij})_{n \geq i \geq j \geq 1} \in \mathbb{R}^{n(n+1)/2}$ of real numbers (see Figure~\ref{fig:gt-array-aij-specified}) satisfying the following (``North-East'' and ``South-East'') inequalities for all $1\leq j < i \leq n$: $NE_{ij}(A) = a_{ij} - a_{i-1, j} \geq 0$ (marked by red)   and $SE_{ij}(A) = a_{i-1, j} - a_{i, j+1} \geq 0$ (marked by blue). For a given integer partition $\mu = (\mu_1,\mu_2,\ldots,\mu_n)$  with $n$-parts, let $\GT(\mu)$ denote the set of all GT-patterns of size $n$ with fixed bottom row $a_{ni} = \mu_i$ for $1 \leq i \leq n$. Note that $\GT(\mu)$ is a polytope and it is called the GT-polytope of shape $\mu$. The \textit{weight} of a GT-pattern $A$ of size $n$ is defined as follows: $$wt(A):=(p_1,p_2,\ldots,p_n) \text{ where } p_i = \sum_{j=1}^i a_{ij} -\sum_{j=1}^{i-1} a_{i-1,j}. $$

For example, let $\mu = (5,4,2,1,0)$ and the following be a GT pattern of shape $\mu$,

\begin{center}
\begin{tikzpicture}[x={(1cm*0.5,-\rootthree cm*0.5)},y={(1cm*0.5,\rootthree cm*0.5)}]
  \draw(0,0)node(a00){$5$};\draw(0,1)node(a01){$5$};\draw(0,2)node(a02){$5$};\draw(0,3)node(a03){$3$};\draw(0,4)node(a04){$3$};
  \draw(1,1)node(a11){$4$};\draw(1,2)node(a12){$4$};\draw(1,3)node(a13){$3$};\draw(1,4)node(a14){$2$};
  \draw(2,2)node(a22){$2$};\draw(2,3)node(a23){$2$};\draw(2,4)node(a24){$1$};
  \draw(3,3)node(a33){$1$};\draw(3,4)node(a34){$0$};
  \draw(4,4)node(a44){$0$};

\end{tikzpicture}
\end{center}
the weight of the above GT-pattern is $(3,2,4,2,1)$.

\subsection{} The character of $V(\mu)$ defined by $\textbf{Char}V(\mu) = \sum_{\beta \in P} d_\mu(\beta)\exp^\beta $ where $d_\mu(\beta) = dim \ V(\mu)_\beta$ is the dimension of the $\beta$-weight space.  Let $\GTz(\mu)$ denote the set of all integral points in the polytope $\GT(\mu)$. For a given $\mu \in P^+$, the set of integral GT-patterns $\GTz(\mu)$ of shape $\mu$ is a combinatorial model for the irreducible module $V(\mu)$ and $\textbf{Char}V(\mu) = \sum_{P \in \GTz(\mu)}\exp^{wt(P)}$. By work of Littelmann \cite{litt:inv} \cite{LittelmannAnnals} it is well known that \[ \GTz(\lambda) \times \GTz(\mu) = \{ (P,Q) \ | \ P \in \GTz(\lambda), \ Q \in \GTz(\mu) \}\] is a combinatorial model for the tensor product $V(\lambda) \otimes V(\mu)$ in the sense that: 
\begin{equation}
    \textbf{Char} \ V(\lambda)\otimes V(\mu) =  \sum_{(P,Q) \in \GTz(\lambda) \times \GTz(\mu)} \exp^{wt(P)+wt(Q)}.
\end{equation}

For given $\lambda,\mu \in P^+$ and a permutation $w \in S_n$, we determine a subset $\GTz(\lambda,w,\mu)$ of $\GTz(\lambda) \times \GTz(\mu)$ and we show that $\GTz(\lambda,w,\mu)$ is a combinatorial model for the KK-module $K(\lambda,w,\mu)$ (see Theorem \ref{thm:gt-mod}).

\section{Gelfand-Tsetlin crystal for the KK-module}

\subsection{} We recall here a binary operation $*$ on the permutation group $S_n$ by Deodhar \cite{deo:ca:87} (more generally on the Weyl group). Let $w \in S_n$ and $s$ be a simple reflection in $S_n$ then $w*s := \textbf{max}\{w,ws\}$, where $\textbf{max}\{w,ws\} = w$ if $w \geq ws$ and $\textbf{max}\{w,ws\} = ws$ if $ws \geq w$ in the Bruhat order. The binary operation $*$ is an associative operation, then for any $w,w' \in S_n$ the permutation $w*w'$ is well defined and $w*w' = \textbf{max} \ I(w)I(w')$ \cite[\S 2 Remark 2.9]{KRV}, where $I(w)$ is the Bruhat interval of $w$, that is $I(w) = \{u \in S_n \ | \ u \leq w \}$. 

\begin{definition}[Demazure product]\label{def:dem-prod}
Let $w = s_{i_1}s_{i_2}\cdots s_{i_k}$ be a word which may not be reduced in the simple reflections $s_{i_j}$ for $1 \leq j \leq k$. The Demazure product of $w$ is defined as follows: \[*(w):= s_{i_1}*s_{i_2}*\cdots*s_{i_k}. \]
\end{definition}
For example $*(s_1s_3s_1s_2s_2) = s_3s_1s_2$.

\begin{lemma}\label{lem:maxdeo-str}
    Let $w = s_{i_1}s_{i_2}\cdots s_{i_k}$ be a word which may not be reduced then the set $\{ *(u) \ | \ u \text{ is a subword of } w \}$ has a unique maximal element and the unique maximal element is $*(w)$. 

\end{lemma}
\begin{proof}
    Let $u = s_{i_{j_1}}s_{i_{j_2}}\cdots s_{i_{j_t}}$ be a subword of $w$ where $j_s \in \{1,2,\ldots,k \}$ for $s=1,2,\ldots,t$. We know that $*(w) = \textbf{max} \ I(s_{i_1})I(s_{i_2})\cdots I(s_{i_k})$, similarly $*(u) = \textbf{max} \ I(s_{i_{j_1}})I(s_{i_{j_2}})\cdots I(s_{i_{j_t}})$. The above expression implies that $*(u)\leq *(w)$.
\end{proof}

\begin{lemma}\label{lem:dem}
Let $w = s_{i_1}s_{i_2}\cdots s_{i_k}$ be a word which may not be reduced, where $s_{i_j}$ are simple reflections for $1 \leq j \leq k$. Consider $w' = s_{i_1}\cdots s_{i_t}$ and $w'' = s_{i_{t+1}}\cdots s_{i_k}$ for some $t$ such that $1 < t < k$. Then, $*(w) = *(w')*(*(w''))$.
\end{lemma}

The above Lemma follows since $*$ is an associative operation.

\subsection{} Now we will state the main theorem. For that, we need the following subset of GT polytope. For a given pair $(P,Q) \in \GT(\lambda)\times \GT(\mu)$ we associate a permutation $\lie{p}(P,Q) \in S_n$ below in equation \eqref{eq:per-asso-GT}. Next we define a subset $\GT(\lambda,w,\mu)$ of $\GT(\lambda)\times \GT(\mu)$ for a given permutation $w$ as follows:
\begin{equation}
    \GT(\lambda,w,\mu):= \{(P,Q) \in \GT(\lambda)\times \GT(\mu) \ | \ \lie{p}(P,Q) \leq w \}
\end{equation}
where $\leq$ is the Bruhat order in $S_n$. We denote by $\GTz(\lambda,w,\mu)$ the set of all integral points in $\GT(\lambda,w,\mu)$. Now we state our main theorem which provides a Gelfand-Tsetlin crystal for KK-modules.

\begin{theorem}\label{thm:gt-mod}

For a given $\lambda,\mu \in P^+$ and $w \in S_n$, the set $\GTz(\lambda,w,\mu)$ is a Gelfand-Tsetlin model for the KK-module $K(\lambda,w,\mu)$ in the sense that 

\begin{equation}
    \textbf{Char} \ K(\lambda,w,\mu) =  \sum_{(P,Q) \in \GTz(\lambda,w,\mu)} \exp^{wt(P)+wt(Q)},
\end{equation}
and $\GTz(\lambda,w,\mu)$ is $e_i,f_i$ invariant where $e_i$ and $f_i$ are raising and lowering crystal operators defined below in the section \ref{sec:rootop-crys}. 

\end{theorem}

We will prove theorem \ref{thm:gt-mod} at the end of the section \ref{sec:proof-main-thm}.

\begin{figure}
\begin{center}
\begin{tikzpicture}[x={(1cm*0.5,-\rootthree cm*0.5)},y={(1cm*0.5,\rootthree cm*0.5)}]
\foreach\i in{0,...,4}
\foreach\j in{\i,...,4}{
  \pgfmathtruncatemacro{\k}{5 - \j + \i};
  \pgfmathtruncatemacro{\l}{\i + 1};
  \draw(\i,\j)node(a\i\j){$a_{\k\l}$};
}
\foreach\i/\ii in{0/-1,1/0,2/1,3/2,4/3}
 \foreach\j/\jj in{0/-1,1/0,2/1,3/2,4/3}{
   \ifnum\i>\j\else\ifnum\i>0 \draw[color=blue](a\ii\j)--(a\i\j);\fi\fi
 }
 \node at (3.7,4.3) {$s_4$};  \node at (2.7,4.3) {$s_3$}; \node at (1.7,4.3) {$s_2$}; \node at (.7,4.3) {$s_1$};
 
 \node at (2.7,3.3) {$s_3$};  \node at (1.7,3.3) {$s_2$}; \node at (.7,3.3) {$s_1$}; 
 
 \node at (1.7,2.3) {$s_2$}; \node at (.7,2.3) {$s_1$};
 
 \node at (.7,1.3) {$s_1$};
\end{tikzpicture}\qquad \qquad
\begin{tikzpicture}[x={(1cm*0.5,-\rootthree cm*0.5)},y={(1cm*0.5,\rootthree cm*0.5)}]
\foreach\i in{0,...,4}
\foreach\j in{\i,...,4}{
  \pgfmathtruncatemacro{\k}{5 - \j + \i};
  \pgfmathtruncatemacro{\l}{\i + 1};
  \draw(\i,\j)node(a\i\j){$a_{\k\l}$};
}
  \foreach\i/\ii in{0/-1,1/0,2/1,3/2,4/3}
 \foreach\j/\jj in{0/-1,1/0,2/1,3/2,4/3}{
  \ifnum\i<\j \draw[color=red](a\i\jj)--(a\i\j); \fi
 }
 \node at (-.3,.3) {$s_4$}; \node at (-.3,1.3) {$s_3$}; \node at (-.3,2.3) {$s_2$}; \node at (-.3,3.3) {$s_1$};
 
 \node at (.7,1.3) {$s_3$}; \node at (.7,2.3) {$s_2$}; \node at (.7,3.3) {$s_1$}; 
 
 \node at (1.7,2.3) {$s_2$}; \node at (1.7,3.3) {$s_1$}; 
 
 \node at (2.7,3.3) {$s_1$};
\end{tikzpicture}
\caption{The blue edges $a_{i-1,j} \longrightarrow a_{i,j+1}$ labelled by $s_{j}$ and the red edges $a_{ij} \longrightarrow a_{i-1,j}$ are labelled by $s_{i-j}$. }
\label{fig:gt-aij-specified}
\end{center}
\end{figure}
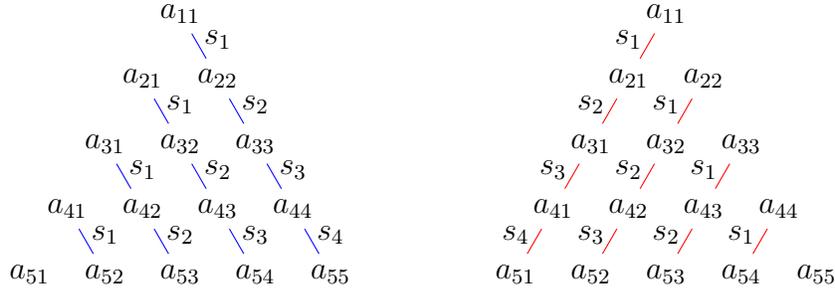

\subsection{Association of a permutation to a GT pair (P,Q)}\label{subsec:asso-perm}\ 

Let $(P,Q) \in \GT(\lambda) \times \GT(\mu)$ for given partitions $\lambda,\mu$ with $n$-parts. Let $P = (a_{ij})_{n\geq i \geq j \geq 1}$ be a GT-pattern of size $n$ as shown in Figure \ref{fig:gt-array-aij-specified}. Consider the labeling of South-East edges of $P$ (marked by blue) $a_{i-1,j} \to a_{i,j+1}$ by the simple reflection $s_{j}$ as shown in the left in Figure \ref{fig:gt-aij-specified} for $n \geq i>j\geq 1$. Let $Q = (a_{ij})_{n\geq i \geq j \geq 1}$ be a GT-pattern of size $n$ as shown in Figure \ref{fig:gt-array-aij-specified}. Consider the labeling of North-East edges of $Q$ (marked by red) $a_{ij} \to a_{i-1,j}$ by the simple reflection $s_{i-j}$ as shown in the right in Figure \ref{fig:gt-aij-specified} for $n \geq i>j\geq 1$.   

 Having done this, we give a procedure to associate a word $w(P,Q)$ in the following two steps. Simultaneously, we will do an example given in Figure \ref{fig:word-exp} for a better understanding of the procedure. Consider the left GT-pattern is $P$ and the right GT-pattern is $Q$ in Figure \ref{fig:word-exp}.
 \begin{procedure}\label{proced}
 \begin{enumerate}
     \medskip \item\label{prcdr1} First read only those labels in $P$ for which $a_{i-1,j} = a_{i,j+1}$ for $n \geq i > j\geq 1$, from bottom to top and left to right, write them from left to right and denote this word by $\lie f (P)$. 
     
     In the example of Figure \ref{fig:word-exp}, we read simple reflections along the blue path which moves from bottom to top and the resulting word is $\lie f (P)=s_1s_2s_4s_2s_1s_1$.
     \medskip \item\label{prcdr2} Next, read only those labels in $Q$ for which $a_{i,j} = a_{i-1,j}$ for $n \geq i>j\geq 1$, from top to bottom and left to right, write them 
     from left to right and denote this word by $\lie i (Q)$. 
     
     In the example of Figure \ref{fig:word-exp}, we read simple reflections along the red path which moves from top to bottom and the resulting word is $\lie i (Q)=s_1s_3s_2s_4s_2$.
 \end{enumerate}
 \end{procedure}
The associated word to the pair $(P,Q)$ is $w(P,Q) := \lie f (P) \lie i (Q)$. In our example of Figure \ref{fig:word-exp} we have $w(P,Q) = s_1s_2s_4s_2s_1s_1 \ s_1s_3s_2s_4s_2$.

The permutation associated to the pair $(P,Q)$ is: 
\begin{equation}\label{eq:per-asso-GT}
    \lie{p}(P,Q) := *(w(P,Q)) \cdot w_0
\end{equation}
where $*(\cdot)$ is the Demazure product (see \ref{def:dem-prod}) and $w_0$ is the longest length element in $S_n$. In our example $*(w(P,Q)) = s_1s_4s_2s_1s_3s_4s_2$ and $\lie{p}(P,Q) = s_2s_1s_3$.

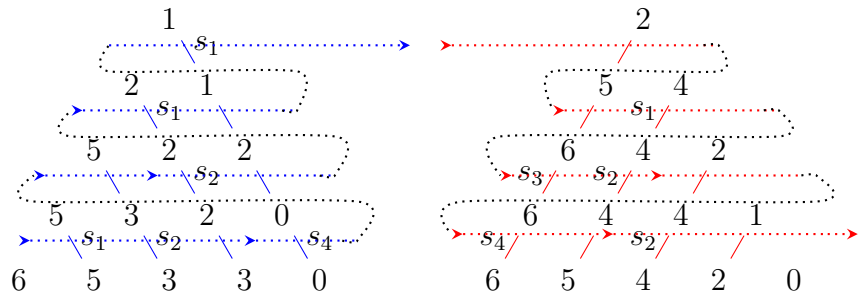
\begin{figure}[h]
\begin{center}
\begin{tikzpicture}[x={(1cm*0.5,-\rootthree cm*0.5)},y={(1cm*0.5,\rootthree cm*0.5)}]
  \draw(0,0)node(a00){$6$};\draw(0,1)node(a01){$5$};\draw(0,2)node(a02){$5$};\draw(0,3)node(a03){$2$};\draw(0,4)node(a04){$1$};
  \draw(1,1)node(a11){$5$};\draw(1,2)node(a12){$3$};\draw(1,3)node(a13){$2$};\draw(1,4)node(a14){$1$};
  \draw(2,2)node(a22){$3$};\draw(2,3)node(a23){$2$};\draw(2,4)node(a24){$2$};
  \draw(3,3)node(a33){$3$};\draw(3,4)node(a34){$0$};
  \draw(4,4)node(a44){$0$};
  
  \foreach\i/\ii in{0/-1,1/0,2/1,3/2,4/3}
 \foreach\j/\jj in{0/-1,1/0,2/1,3/2,4/3}{
  \ifnum\i>\j\else\ifnum\i>0 \draw[color=blue](a\ii\j)--(a\i\j);\fi\fi
 }
 \node at (3.7,4.3) {$s_4$}; 
 \node at (.7,4.3) {$s_1$};
 
\node at (1.7,3.3) {$s_2$}; \node at (.7,3.3) {$s_1$}; 
 
 \node at (1.7,2.3) {$s_2$};
 
 \node at (.7,1.3) {$s_1$};
 
 \draw[thick, blue, stealth reversed - , style = dotted] (-.3,.3)--(2.7,3.3);
 \draw[thick, blue, stealth reversed - , style = dotted] (2.7,3.3)--(4.2,4.8);
 \draw[thick, style = dotted] plot [smooth, tension=0.4] coordinates {(4,4.6) (4.2,4.8) (3.9,5.1) (-.4,.8) (-.6,1)};
 \draw[thick, blue, stealth reversed - , style = dotted] (-.6,1)--(.9,2.5);
 \draw[thick, blue, stealth reversed - , style = dotted] (.9,2.5)--(3.4,5);
 \draw[thick, style = dotted] plot [smooth, tension=0.4] coordinates {(3.2,4.8) (3.4,5) (3.1,5.3) (-.4,1.8) (-.6,2)};
 \draw[thick, blue, stealth reversed - , style = dotted] (-.6,2)--(2.4,5);
 \draw[thick, style = dotted] plot [smooth, tension=0.4] coordinates {(2.2,4.8) (2.4,5) (2.1,5.3) (-.4,2.8) (-.6,3)};
 \draw[thick, blue, stealth reversed - , -stealth, style = dotted] (-.6,3)--(3.4,7);
\end{tikzpicture}
\begin{tikzpicture}[x={(1cm*0.5,-\rootthree cm*0.5)},y={(1cm*0.5,\rootthree cm*0.5)}]
  \draw(0,0)node(a00){$6$};\draw(0,1)node(a01){$6$};\draw(0,2)node(a02){$6$};\draw(0,3)node(a03){$5$};\draw(0,4)node(a04){$2$};
  \draw(1,1)node(a11){$5$};\draw(1,2)node(a12){$4$};\draw(1,3)node(a13){$4$};\draw(1,4)node(a14){$4$};
  \draw(2,2)node(a22){$4$};\draw(2,3)node(a23){$4$};\draw(2,4)node(a24){$2$};
  \draw(3,3)node(a33){$2$};\draw(3,4)node(a34){$1$};
  \draw(4,4)node(a44){$0$};
  
  \foreach\i/\ii in{0/-1,1/0,2/1,3/2,4/3}
 \foreach\j/\jj in{0/-1,1/0,2/1,3/2,4/3}{
  \ifnum\i<\j \draw[color=red](a\i\jj)--(a\i\j); \fi
 }
  
  \node at (-.3,.3) {$s_4$}; \node at (-.3,1.3) {$s_3$}; 
 
\node at (.7,2.3) {$s_2$}; \node at (.7,3.3) {$s_1$}; 
 
 \node at (1.7,2.3) {$s_2$};
 
 \draw[thick, red, stealth reversed - , style = dotted] (-2.5,1.1)--(1.2,4.8);
 \draw[thick, style = dotted] plot [smooth, tension=0.4] coordinates {(1,4.6) (1.2,4.8) (1.4,4.6) (-.8,2.4) (-.5,2.1)};
 \draw[thick, red, stealth reversed - , style = dotted] (-.5,2.1)--(2.5,5.1);
 \draw[ thick, style = dotted] plot [smooth, tension=0.4] coordinates {(2.3,4.9) (2.5,5.1) (2.7,4.9) (-1,1.2) (-.7,.9)};
 \draw[ thick, red, stealth reversed - , style = dotted] (-.7,.9)--(1.3,2.9);
 \draw[thick, red, stealth reversed - , style = dotted] (1.3,2.9)--(3.3,4.9);
 \draw[thick, style = dotted] plot [smooth, tension=0.4] coordinates {(3.3,4.9) (3.5,5.1) (3.7,4.9) (-1,.2) (-.9,-.2)};
 \draw[thick, red, stealth reversed - , style = dotted] (-.9,-.2)--(1.1,1.8);
 \draw[ thick, red, stealth reversed - , style = dotted] (1.1,1.8)--(4.1,4.8);
 \draw[thick, red, -stealth, style = dotted] (4.1,4.8)--(4.5,5.2);
 

\end{tikzpicture}
\end{center}
\caption{Associated word is $w(P,Q) = s_1s_2s_4s_2s_1s_1 \ s_1s_3s_2s_4s_2$.}
\label{fig:word-exp}
\end{figure}

The following is an immediate Corollary of theorem \ref{thm:gt-mod}.
\begin{corollary}
For given $\lambda,\mu \in P^+$, let $\GTz(\lambda+\mu)$ be the collection of all pairs $(P,Q)$ in $\GTz(\lambda) \times \GTz(\mu)$ such that the word $w(P,Q)$ contains $w_0$ as a subword then, \[ \textbf{char } V(\lambda+\mu) = \sum_{(P,Q) \in \GTz(\lambda+\mu)} \exp^{wt(P)+wt(Q)}. \]
\end{corollary}


\section{Lakshmibai-Seshadri path model for the KK-modules}\label{sec:lspath-mod}

\subsection{} We recall the Lakshmibai-Seshadri (LS) path model from \cite{KRV}. LS path introduced by Littelmann \cite{litt:inv}, it is a piecewise linear path $\pi:[0,1] \to P_{\mathbb{R}}$ such that $\pi(0)=0$ and $\pi(1) \in P$, where $P_{\mathbb{R}}$ is the real vector space $P \otimes_{\mathbb{Z}} \mathbb{R}$.  Let $\lambda \in P^+$, an LS path $\pi$ of shape $\lambda$ is given by a sequence \[ \pi = (w_1 > w_2 > \cdots >w_r;0 < a_1 < a_2 < \cdots < a_r=1)\] which satisfies ``\textit{chain conditions}", where $w_i \in W/W_\lambda$ and $W_\lambda$ is the stabilizer of $\lambda$. 
The cosets $w_1$ and $w_r$ are called the initial direction and final direction of $\pi$ respectively. Let $P_\lambda$ denote the set of all LS paths of shape $\lambda$.

Let $\pi \ostar \pi'$ denote the concatenation of two LS paths $\pi$ and $\pi'$ and \[ P_\lambda \ostar P_\mu = \{ \pi \ostar \pi' \ | \ \pi \in P_\lambda, \ \pi' \in P_\mu \}.\] Note that we can identify $P_\lambda \ostar P_\mu$ with $P_\lambda \times P_\mu$ as a set by assuming $\pi \ostar \pi' = (\pi,\pi')$. We associate a Weyl group element $\lie{w}(\pi \ostar \pi')$ to each element $\pi \ostar \pi' \in P_\lambda \ostar P_\mu$ as follows (see \cite{KRV}, section 3.1):
\begin{equation}\label{eq:wasso}
    \lie{w}(\pi \ostar \pi'):= \textbf{min}\{W_\lambda \ I(\tau^{-1}) \ \varphi \ W_\mu\}
\end{equation}
where $\tau$ is any representative of the final direction of $\pi$ and $\varphi$ is any representative of the initial direction of $\pi'$.

\begin{definition}[KK set \cite{KRV}]\label{def:kkset}
For given $\lambda,\mu \in P^+$ and $w \in W$, the KK set defined as follows:
\begin{equation}
    P(\lambda,w,\mu) = \{ \pi \ostar \pi' \in P_\lambda \ostar P_\mu \ | \ \lie{w}(\pi \ostar \pi') \leq w \}.
\end{equation}
\end{definition}

The following theorem shows that the KK set $P(\lambda,w,\mu)$ is a path model for the KK-module $K(\lambda,w,\mu)$.

\begin{theorem}[\cite{KRV}]\label{thm:lspath-mod}
Let $\lie g$ be a symmetrizable Kac-Moody Lie algebra which is either symmetric or of finite type, then the KK set $P(\lambda,w,\mu)$ is a path model for the KK-module $K(\lambda,w,\mu)$ in the sense that
\begin{equation}
    \textbf{char} \  K(\lambda,w,\mu) = \sum_{\pi \ostar \pi' \in P(\lambda,w,\mu)} \exp^{\pi(1)+\pi'(1)},
\end{equation}
and $K(\lambda,w,\mu)$ is crystal operators $e_i,f_i$ invariant.
\end{theorem}

\subsection{} We will simplify the definition of the Weyl group element $\lie{w}(\pi \ostar \pi')$. First, we recall some corollaries and lemmas from \cite{KRV} which are needed to prove this simplification. 


\begin{lemma}[\cite{KRV}, Corollary 2.5]\label{lem:uqmin}
Let $K$ be a subset of Weyl group $W$ that admits a unique minimal element say $u$. Then, for any two elements $\sigma_1$ and $\sigma_2$ in $W$, the set $I(\sigma_1)K I(\sigma_2)$ has a unique minimal element and this element is the unique minimal element in $I(\sigma_1)u I(\sigma_2)$.
\end{lemma}

\begin{lemma}[\cite{KRV}, Remark 2.4]\label{lem:maxstar}
Let $\sigma$ and $w$ be elements of $W$. The set $\{I(\sigma)w\}$ has a unique maximal element and $$ \textbf{max} \ \{I(\sigma)w \} = \sigma * w .$$
\end{lemma}

\begin{lemma}\label{lem:maxmin}
Let $K$ be a subset of $W$ which admits a unique minimum, then $$\textbf{min} \ \{ u \ | \ u \in K \} = (\textbf{max} \ \{ u w_0 \ | \ u \in K \})\cdot w_0,$$
where $w_0$ is the longest element in $W$.
\end{lemma}

The above lemma follows since $u \leq v$ if and only if $vw_0 \leq uw_0$ in the Bruhat order.

\begin{lemma}[\cite{KRV}, Remark 2.9]\label{lem:ww}
Let $w$ and $w'$ be two element in $W$ then, $I(w)I(w') = I(w*w')$. Moreover, if $\ell(ww')=\ell(w)+\ell(w')$ then $I(w)I(w') = I(ww')$, where $\ell(\cdot)$ is the length function on $W$.
\end{lemma}

The following proposition is a simplification of $\lie{w}(\pi \ostar \pi')$ defined in equation \eqref{eq:wasso}.

\begin{proposition}\label{lem:wasso}
For given an element $\pi \ostar \pi' \in P_\lambda \ostar P_\mu$, the Weyl group element associated in the equation \eqref{eq:wasso} is:
\begin{equation}
    \lie{w}(\pi \ostar \pi') = (\widehat{\tau}^{-1}*\widetilde{\varphi}w_0)\cdot w_0
\end{equation}
where $\widehat{\tau}$ is the maximal representative of the final direction of $\pi$ and $\widetilde{\varphi}$ is the minimal representative of the initial direction of $\pi'$.
\end{proposition}

\begin{proof}
For given $\pi \ostar \pi' \in P_\lambda \ostar P_\mu$ the permutation $\lie{w}(\pi \ostar \pi')$ defined in equation \eqref{eq:wasso} is
\begin{equation}\label{eq:wasso-lem}
    \lie{w}(\pi \ostar \pi'):= \textbf{min}\{W_\lambda \ I(\tau^{-1}) \ \varphi \ W_\mu\}
\end{equation}
where $\tau$ is any representative of the final direction of $\pi$ and $\varphi$ is any representative of the initial direction of $\pi'$. Let $w_{0\lambda}$ be the longest element in the stabilizer $W_\lambda$. Then equation \eqref{eq:wasso-lem} can be written as follows:
\begin{equation}\label{eq:wasso-I}
    \lie{w}(\pi \ostar \pi'):= \textbf{min}\{I(w_{0\lambda}) \ I(\tau^{-1}) \ \varphi \ W_\mu\}.
\end{equation}
By lemma \ref{lem:uqmin} we can write equation \eqref{eq:wasso-I} as
\begin{align}
    \textbf{min}\{I(w_{0\lambda}) \ I(\tau^{-1}) \ \varphi \ W_\mu \} & =  \textbf{min}\{I(w_{0\lambda}) \ \textbf{min} \{ \ I(\tau^{-1}) \  \textbf{min} \{\varphi \ W_\mu \} \}\}\\
    & = \textbf{min}\{I(w_{0\lambda}) \ \textbf{min} \{ \ I(\tau^{-1}) \  \widetilde{\varphi} \}\}\\
    & = \textbf{min}\{I(w_{0\lambda})  \ I(\tau^{-1}) \  \widetilde{\varphi} \}\label{eq:11},
\end{align}

since $\textbf{min} \{\varphi \ W_\mu \} = \widetilde{\varphi}$. By lemma \ref{lem:ww}, equation \eqref{eq:11} can be written as follows,
\begin{equation}\label{eq:12}
  \textbf{min}\{I(w_{0\lambda})  \ I(\tau^{-1}) \  \widetilde{\varphi} \}  = \textbf{min}\{I(w_{0\lambda}*\tau^{-1}) \  \widetilde{\varphi} \}. 
\end{equation}

We can consider $\tau$ is the minimal representative of coset $\tau W_\lambda$ then $\widehat{\tau} = \tau w_{0\lambda}$ is the maximal coset representative of $\tau W_\lambda$. Then, equation \eqref{eq:12} is:
\begin{equation*}
    \textbf{min}\{I(w_{0\lambda}*\tau^{-1}) \  \widetilde{\varphi} \}   = \textbf{min}\{I(w_{0\lambda}\tau^{-1}) \  \widetilde{\varphi} \}
    =\textbf{min}\{I((\tau w_{0\lambda})^{-1}) \  \widetilde{\varphi} \}  =\textbf{min}\{I(\widehat{\tau}^{-1}) \  \widetilde{\varphi} \}.
\end{equation*}
By lemma \ref{lem:maxmin} $ \textbf{min}\{I(\widehat{\tau}^{-1}) \  \widetilde{\varphi} \} =(\textbf{max}\{I(\widehat{\tau}^{-1}) \  \widetilde{\varphi} w_0 \})\cdot w_0$, and by lemma \ref{lem:maxstar} we have $$\textbf{max}\{I(\widehat{\tau}^{-1}) \  \widetilde{\varphi} w_0 \} = \widehat{\tau}^{-1}*\widetilde{\varphi}w_0.$$
\end{proof}

\section{Demazure crystals and Kogan faces}\label{sec:rootop-crys}

\subsection{} We will recall the definition of \textit{crystal} and \textit{crystal isomorphism} from \cite{kashiwara}.
Let $\lie g$ be a symmetrizable Kac-Moody Lie algebra, $\lie h$ be a Cartan subalgebra, and $\lie{h}^*$ be the dual subspace of $\lie h$. Denote by $\langle \cdot , \cdot \rangle : \lie h^* \times \lie{h} \to \mathbb{C}$ the canonical pairing and by $P$ the weight lattice of $\lie g$. Let $\{ \alpha_i \ | \ i \in I \} \subset \lie{h}^*$ be the set of simple roots and $\{ \alpha_i^\vee \ | \ i \in I \} \subset \lie{h}$ be the set of simple coroots, where $I = \{1,2,\ldots,n-1\}$ is an indexing of the vertices of the Dynkin diagram. 

\begin{definition}[Crystal \cite{kashiwara}]\label{def:crys}
A crystal is a set $\lie B$ equipped with the following maps:
\begin{itemize}
    \item $wt: \lie B \to P$.
    \item $\psi_i : \lie B \to \mathbb{Z} \cup \{ - \infty \}$ and $\varphi_i : \lie B \to \mathbb{Z} \cup \{ - \infty \}$ for $i \in I$.
    \item $e_i : \lie B \to \lie B \cup \{ 0 \}$ and $f_i : \lie B \to \lie B \cup \{ 0 \}$ for $i \in I$.
\end{itemize}
where $- \infty$ and $0$ are additional elements that are not contained in $\mathbb{Z}$ and $\lie B$, respectively, such that the following holds: for $i \in I$ and $b \in \lie B$,

\begin{enumerate}
    \item $\varphi_i(b) = \psi_i(b) + \langle wt(b) , \alpha_i^\vee \rangle$.
     \item $wt(e_i(b)) = wt(b) + \alpha_i$, $\psi_i(e_i(b)) = \psi_i(b) -1$ and $\varphi_i(e_i(b)) = \varphi_i(b)+1$, if $e_i(b) \in \lie B$.
     \item $wt(f_i(b)) = wt(b) - \alpha_i$, $\psi_i(f_i(b)) = \psi_i(b) + 1$ and $\varphi_i(f_i(b)) = \varphi_i(b)-1$, if $f_i(b) \in \lie B$.
    \item $e_i(b) = b'$ for $b' \in \lie B$, if and only if $f_i(b') = b$ for $b \in \lie B$.
     \item $e_i(b) = f_i (b) = 0$, if $\varphi_i(b) = - \infty$.
\end{enumerate}
\end{definition}

\begin{definition}[Crystal Isomorphism \cite{kashiwara}]\label{def:crysiso}
Let $\lie B_1$ and $\lie B_2$ be two crystals. A map $\phi:\lie B_1 \cup \{ 0 \} \to \lie B_2 \cup \{ 0 \} $ is called crystal isomorphism if $\phi$ is bijective and satisfies the followings:
\begin{enumerate}
     \item $\phi(0)=0$.
     \item  $wt(b) = wt(\phi(b))$, $\psi_i(b) = \psi_i(\phi(b))$ and $\varphi_i(b) = \varphi_i(\phi(b))$ for all $i \in I$ and $b \in \lie B_1$ such that $\phi(b) \in \lie B_2$.
     \item $\phi(e_i(b)) = e_i (\phi(b))$ and $\phi (f_i(b)) = f_i (\phi(b))$ for all $b \in \lie B_1$.
    \end{enumerate}
    where we set $e_i(\phi(b)) = f_i(\phi(0)) = 0$, if $\phi(b) = 0$.
\end{definition}

Littelmann has shown that the set of LS paths $P_\mu$ of shape $\mu$ for $\mu \in P^+$ is a crystal. For the definition of crystal operators $e_i$ and $f_i$ for $i \in I$ on LS paths we give references of Littelmann \cite{litt:inv} and \cite{LittelmannAnnals} which he calls the \textit{root operators}.

\subsection{Crystal operators on GT patterns}\label{ss:cryopongt} For a given integer partition $\mu = (\mu_1,\mu_2,\ldots,\mu_n)$ the set of integral GT-patterns $\GTz(\mu)$ of shape $\mu$ is also a crystal. We will recall the definition of crystal operators $e_i$ and $f_i$ for $i \in I=\{1,2,\ldots,n-1\}$ on $\GTz(\mu)$.

Let $P$ be the following GT pattern of shape $\mu$.

\begin{center}
\begin{tikzpicture}[x={(1cm*0.5,-\rootthree cm*0.5)},y={(1cm*0.5,\rootthree cm*0.5)}]
  \draw(0,0)node(a00){$a_{n1}$};\draw(0,1)node(a01){$\cdot$};\draw(0,2)node(a02){$a_{31}$};\draw(0,3)node(a03){$a_{21}$};\draw(0,4)node(a04){$a_{11}$};
  \draw(1,1)node(a11){$a_{n2}$};\draw(1,2)node(a12){$\cdot$};\draw(1,3)node(a13){$a_{32}$};\draw(1,4)node(a14){$a_{22}$};
  \draw(2,2)node(a22){$\cdot$};\draw(2,3)node(a23){$\cdot$};\draw(2,4)node(a24){$a_{33}$};
  \draw(3,3)node(a33){$\cdot$};\draw(3,4)node(a34){$\cdot$};
  \draw(4,4)node(a44){$a_{nn}$};

\end{tikzpicture}
\end{center}

where $a_{n,i} = \mu_i$ for $i=1,2,\ldots,n.$

For a fixed $i \in I$ we will recall the definition of $e_i$ and $f_i$. Define the integers $d_t$ for $1 \leq t \leq i+1$ inductively as follows, set  $d_1 :=  a_{i,1} - a_{i+1,1}$ and \[ d_t := d_{t-1}+ a_{i,t-1}+a_{i,t} - a_{i-1,t-1} - a_{i+1,t} \ \text{for} \ 1 < t \leq i+1,\] where assume that $a_{k,k+1} = 0$ for $k \in \{ i-1,i \}$. Note that $d_1 \leq 0$ since $a_{i,1} \leq a_{i+1,1}$. 
Let $d = \textbf{min}\{d_1,d_2,\ldots,d_{i+1}\}$, $m$ be the smallest integer such that $d_m = d$ and $M$ be the largest integer such that $d_M=d$.

The action of $e_i$ on $P$ defined as follows:
\begin{enumerate}
     \item if $d_m = 0$ then, $e_i(P)=0$.
    \item if $d_m < 0$ then, $e_i(P) = P'$ where $P'$ have all entries same as $P$ except the entry $a_{im}$ will change to $a_{im}+1$.
\end{enumerate}

The action of $f_i$ on $P$ defined as follows,
\begin{enumerate}
     \item if $M = i+1$ then, $f_i(P)=0$.
    \item if $M < i+1$ then, $f_i(P) = P'$ where $P'$ have all entries same as $P$ except the entry $a_{iM}$ will change to $a_{iM}-1$.
\end{enumerate}

The following proposition \ref{prop:crysiso} is a well-known fact and a variation of proof can be seen in \cite{LittelmannAnnals}.

\begin{proposition}\label{prop:crysiso}
There exists a crystal isomorphism $ \phi: \GTz(\mu) \to P_\mu$. 
\end{proposition}

\subsection{} For an element $G \in \GTz(\mu)$, We define the initial and final direction of $G$ as the same as LS path $\phi(G)$. There exists a unique element $G_\mu^0$ in $\GTz(\mu)$ such that $e_i (G_\mu^0) =0 $ for all $i\in I$ and a unique element $G_\mu^*$ in $\GTz(\mu)$ such that $f_i(G_\mu^*) = 0$ for all $i \in I$. 
The element $G_\mu^0$ is called the highest-weight element and $G_\mu^*$ is called the lowest-weight element of $\GTz(\mu)$.

\begin{definition}\label{def:demcry}
   
 Given $w \in S_n$, fix a reduced decomposition $w = s_{i_1} s_{i_2} \cdots s_{i_k}$. The set \[ \demcrys(\mu, w) :=\{f_{i_1}^{m_1} f_{i_2}^{m_2} \cdots f_{i_k}^{m_k}(G^0_\mu) \ | \ m_j \geq 0; \ 1 \leq j \leq k \} \subseteq \GTz(\mu) \] is called a {\em Demazure crystal}. 
 \end{definition}
 
 \begin{definition}\label{def:opdemcry}
   
 Given $w \in S_n$, fix a reduced decomposition $ww_0 = s_{j_1} s_{j_2} \cdots s_{i_t}$. The set \[ \demcrys(\mu, w)^{op} :=\{e_{j_1}^{m_1} e_{j_2}^{m_2} \cdots e_{j_t}^{m_t} G_\mu^* \ | \ m_s \geq 0; \ \ 1 \leq s \leq t \} \subseteq \GTz(\mu) \] is called an {\em opposite Demazure crystal}. 
 \end{definition}

 The following two propositions determine the minimal representative of the initial direction and the maximal representative of the final direction of a given $G \in \GTz(\mu)$ in terms of Demazure and opposite Demazure crystals respectively.

 \begin{proposition}\label{prop:initdem}
     Let $G \in \GTz(\mu)$ and $ S =  \{ w \in S_n \ | \ G \in \demcrys(\mu,w)  \}$. Then the set $S$ has a unique minimal element and $\textbf{min} \ S$ is the minimal representative of the initial direction of $G$.
 \end{proposition}
 \begin{proof}
     From work of Littelmann \cite[\S 5.2]{litt:inv}, Demazure crystals can be defined in terms of initial direction as follows:
     \begin{equation}\label{eq:altdem}
         \demcrys(\mu, w) = \{ T \in \GTz(\mu) \ | \ \text{init}(T) \leq wS_\mu \}
     \end{equation}  where $\text{init}(T)$ denote the initial direction of $T$.
     Let $\widetilde{\varphi}$ be the minimal representative of the initial direction of $G$. From the above expression \eqref{eq:altdem} of Demazure crystal, clearly $\widetilde{\varphi} \in S$ and for any $w \in S$ we have $\text{init}(G) = \widetilde{\varphi} S_\mu \leq w S_\mu$ which implies that $\widetilde{\varphi} \leq w$.
 \end{proof}

 \begin{proposition}\label{prop:finaldem}
     Let $G \in \GTz(\mu)$ and $ \overline{S} =  \{ w \in S_n \ | \ G \in \demcrys(\mu,w)^{op}  \}$. Then the set $\overline{S}$ has a unique maximal element and $\textbf{max } \overline{S}$ is the maximal representative of the final direction of $G$.
 \end{proposition}
  \begin{proof}
     Similar to Demazure crystals, opposite Demazure crystals can be defined in terms of final direction as follows:
     \begin{equation}\label{eq:altopdem}
         \demcrys(\mu, w)^{op} = \{ T \in \GTz(\mu) \ | \ \text{fin}(T) \geq wS_\mu \}
     \end{equation}  where $\text{fin}(T)$ denote the final direction of $T$.
     Let $\widehat{\tau}$ be the maximal representative of the final direction of $G$. From the above expression \eqref{eq:altopdem} of opposite Demazure crystal, clearly $\widehat{\tau} \in \overline{S}$ and for any $w \in \overline{S}$ we have $\text{fin}(G) = \widehat{\tau} S_\mu \geq w S_\mu$ which implies that $\widehat{\tau} \geq w$.
 \end{proof}

Next, we will introduce Kogan faces and dual Kogan faces which establish a connection between Demazure Crystals and opposite Demazure crystals, for more details see \cite{fujita}.

\subsection{Kogan faces of GT polytope}\label{ss:kogan}
Let $\mu \in P^+$ and $\GT(\mu)$ be the GT polytope of shape $\mu$ defined in the section \ref{sec:gtpattern}.
Fix a subset $F \subseteq \{(i,j): n \geq i > j \geq 1\}$. Consider the face of $GT(\mu)$ obtained by setting North-East inequalities $NE_{ij}=a_{ij}-a_{i-1,j}=0$ for $(i,j) \in F$ and leaving all other inequalities untouched.  
We call this the {\em Kogan face} $\kogan(\mu,F)$. To each pair $i > j$, associate the simple transposition $s_{i-j} \in S_n$. We list the elements of $F$ in lexicographically increasing order: $(i,j)$ precedes $(i',j') \iff$ either $i<i'$, or $i=i'$ and $j<j'$. Denote the product of the corresponding $s_{i-j}$ in this order by $\sigma(F)$. If $\len \sigma(F) = |F|$, i.e., this word is reduced, we say that $\kogan(\mu,F)$ is {\em reduced Kogan face} and set \cite[Definition 5.1]{fujita}:
\[ \varpi(F) = \wnot \,\sigma(F) \,\wnot. \]
For $w \in S_n$, let $\kogan(\mu, w) := \cup \kogan(\mu, F)$, the union over reduced Kogan face $\kogan(\mu,F)$ for which $\varpi(F) = w$. Denote the set of integral points in $\kogan(\mu, w) $ by $\koganz(\mu, w) $.

We state the following proposition by Fujita which establishes a connection between Kogan faces and Demazure crystals.

\begin{proposition}[\cite{fujita}, Corollary 5.19]\label{prop:fujita}
$\koganz(\mu, \wnot w) = \demcrys(\mu, w)$. 
\end{proposition}

\subsection{Dual Kogan faces of GT polytope}\label{ss:dualkogan}
Fix a subset $F \subseteq \{(i,j): n \geq i > j \geq 1\}$. Consider the face of $GT(\mu)$ obtained by setting South-East inequalities $SE_{ij}= a_{i-1,j} - a_{i,j+1}=0$ for $(i,j) \in F$ and leaving all other inequalities untouched.  
We call this the {\em dual Kogan face} $\overline{\kogan}(\mu, F)$. To each pair $(i,j) \in \{(i,j): n \geq i > j \geq 1\}$, associate the simple transposition $s_{j} \in S_n$. We consider a total order on pairs $(i,j)$ defined by $(i,j)$ precedes $(i',j') \iff$ either $i<i'$, or $i=i'$ and $j>j'$. We list the elements of $F$ in the increasing order relative to this total order. Denote the product of the corresponding $s_{j}$ in this order by $\overline{\sigma}(F)$. If $\len \overline{\sigma}(F) = |F|$, i.e., this word is reduced, we say that $\overline{\kogan}(\mu, F)$ is a {\em reduced dual Kogan face} and set \cite[Definition 5.1]{fujita}:
\[ \overline{\varpi}(F) = \wnot \,\overline{\sigma}(F) \,\wnot. \]
For $w \in S_n$, let $\overline{\kogan}(\mu, w) := \cup \overline{\kogan}(\mu, F)$, the union over reduced dual Kogan faces $\overline{\kogan}(\mu, F)$ for which $\overline{\varpi}(F) = w$. Denote the set of integral points in $\overline{\kogan}(\mu, w) $ by $\overline{\koganz}(\mu, w) $.

We state the following proposition by Fujita which establishes a connection between dual Kogan faces and opposite Demazure crystals.

\begin{proposition}[\cite{fujita}, Corollary 5.2]\label{prop:fujita1}
 $\overline{\koganz}(\mu, \wnot w \wnot) = \demcrys(\mu, w)^{op}$. 
\end{proposition}

\subsection{BiKogan faces}\label{s:bikogan}
Let $\lambda,\mu$ be two elements of $P^+$. Consider the GT polytopes $\GT(\lambda)$ and  $\GT(\mu)$, with a typical element of the former given by $(a_{ij})_{n\geq i \geq j \geq 1}$ and that of the latter given by $(b_{ij})_{n\geq i \geq j \geq 1}$ as in figure \ref{fig:gt-array-aij-specified}.
Fix subsets $F,F'$ of the set $\{(i,j): n \geq i > j \geq 1\}$. Consider the face of the Cartesian product $\GT(\lambda)\times \GT(\mu)$ of GT polytopes obtained by setting $ a_{i-1, j} - a_{i, j+1} =0$ for all $(i,j) \in F$ and  $b_{ij} - b_{i-1, j} = 0$ for all $(i,j) \in F'$ and leaving all other inequalities untouched.  Observe that the former equalities concern $SE$ differences, while the latter are $NE$ differences.
We call this face a {\em BiKogan face} and denote it by $\kogan_{\lambda,\mu}(F,F')$.

Define an element $\varpi(F,F') \in S_n$ associated to the pair $(F,F')$ as follows: $\varpi(F,F') = \overline{\sigma}(F)^{-1}\cdot \sigma(F')$.
We say the BiKogan face $\kogan_{\lambda,\mu}(F,F')$ is reduced if $\text{len}(\varpi(F,F')) = |F| + |F'|$.
For a given $w \in S_n$, let $$\kogan_{\lambda,\mu}( w) := \cup \kogan_{\lambda,\mu}(F, F'),$$ the union being taken over all reduced BiKogan faces for which $\varpi(F,F') = w$.
\begin{lemma}\label{lem:Bikogan}
    Let $\kogan_{\lambda,\mu}(F,F')$ be reduced for a given pair $(F,F')$. If $(P,Q) \in \kogan_{\lambda,\mu}(F,F')$ then $\varpi(F,F')\leq *(w(P,Q))$.
\end{lemma}
\begin{proof}
    Let $(P,Q) \in \kogan_{\lambda,\mu}(F,F')$ and $P = (a_{ij})_{n\geq i \geq j \geq 1}, \ Q = (b_{ij})_{n\geq i \geq j \geq 1}$. Then $F$ is a subset of $\{ (i,j) \ | \ a_{i-1,j}=a_{i,j+1} \}$ and $F'$ is a subset of $\{ (i,j) \ | \ b_{ij} = b_{i-1,j} \}$. From the item \eqref{prcdr1} and \eqref{prcdr2} of Procedure \ref{proced} observe that $\sigma(F')$ and $\overline{\sigma}(F)^{-1}$ are subwords of $\lie i(Q)$ and $\lie f(P)$ respectively. 
    Then we have $\varpi(F,F') = \overline{\sigma}(F)^{-1}\sigma(F)$ is a subword of the word $\lie{f}(P)\lie i(Q) = w(P,Q)$. Since $\kogan_{\lambda,\mu}(F,F')$ is reduced BiKogan face then $\varpi(F,F') \leq *(w(P,Q))$.
\end{proof}

We state the following theorem which gives another description of $\GT(\lambda,w,\mu)$ in terms of BiKogan faces.

\begin{theorem}\label{thm:Bikogan}
Let $\lambda,\mu$ be two elements in $P^+$ and $w$ be a permutation in $S_n$ then $\kogan_{\lambda,\mu}( w\wnot) = \GT(\lambda,w,\mu)$. In particular, the set of integral points in $\kogan_{\lambda,\mu}( w\wnot)$ is a combinatorial model for the KK-module $K(\lambda,w,\mu)$.
\end{theorem}

\begin{proof}
    Let $(P,Q)\in \kogan_{\lambda,\mu}( w\wnot)$. Then there exist subsets $F,F'$ of the set $\{ (i,j)\ | \ n\geq i > j \geq 1\}$ such that the BiKogan face $\kogan_{\lambda,\mu}(F,F')$ is reduced, $\varpi(F,F') = w\wnot$ and $(P,Q)\in \kogan_{\lambda,\mu}(F,F')$. From the lemma \ref{lem:Bikogan}, $\varpi(F,F') \leq *(w(P,Q))$. Hence $\lie p (P,Q)  = *(w(P,Q))\wnot \leq w$ which implies that $(P,Q)\in \GT(\lambda,w,\mu)$.

    For the converse, let $(P,Q)\in \GT(\lambda,w,\mu)$ then $\lie p (P,Q)  = *(w(P,Q))\wnot \leq w$. Since $*(w(P,Q)) \geq w\wnot$, then we can choose subsets $F,F'$ of the set $\{ (i,j)\ | \ n\geq i > j \geq 1\}$ such that the BiKogan face $\kogan_{\lambda,\mu}(F,F')$ is reduced, $(P,Q)\in \kogan_{\lambda,\mu}(F,F')$ and $\varpi(F,F') = w\wnot$. Hence $(P,Q) \in \kogan_{\lambda,\mu}( w\wnot)$.
\end{proof}

\section{Proof of the theorem \ref{thm:gt-mod}}\label{sec:proof-main-thm}

\subsection{} 
We state the following two lemmas \ref{lem:init} and \ref{lem:final} which are the main ingredients to prove theorem \ref{thm:gt-mod}. For a given $P \in \GTz(\mu)$, these lemmas provide a procedure to determine the minimal representative of the initial direction and the maximal representative of the final direction of $\phi(P)$ respectively.

\begin{lemma}\label{lem:init}
Let $\lie i (P)$ be the word associated in item \eqref{prcdr2} of the procedure \ref{proced} then, the minimal representative of the initial direction of the LS path $\phi(P)$ is $*(\lie i (P)) w_0$.
\end{lemma}

\begin{lemma}\label{lem:final}
Let $\lie f (P)$ be the word associated in item \eqref{prcdr1} of the procedure \ref{proced} then, the maximal representative of the final direction of the LS path $\phi(P)$ is $(*(\lie f (P)))^{-1}$.
\end{lemma}

We will give a proof of lemma \ref{lem:init} and \ref{lem:final} later in this section.

The following two corollaries are immediate results of the above lemmas which gives another characterization of the Demazure crystals and opposite Demazure crystals.

\begin{corollary}
    With the notation in the lemma \ref{lem:init} \[ \demcrys(\mu,w) = \{ P \in \GTz(\mu) \ | \ *(\lie{i}(P))\wnot \leq w \}. \]
\end{corollary}

\begin{corollary}
    With the notation in the lemma \ref{lem:final} \[ \demcrys(\mu,w)^{op} = \{ P \in \GTz(\mu) \ | \ \wnot(*(\lie{f}(P)))^{-1} \leq \wnot w \}. \]
\end{corollary}

\subsection{} The next proposition will give a connection between the LS path model and the GT model for the KK-modules. We will use it to prove theorem \ref{thm:gt-mod}.

\begin{proposition}\label{prop:GT-LS}
The following diagram commutes
\[
\begin{tikzcd}
\GTz(\lambda)\times \GTz(\mu) \arrow[r, "\phi \times \phi"] \arrow[rd, "\lie p"'] & P_\lambda \ostar P_\mu \arrow[d, "\lie w"] \\
 & S_n
\end{tikzcd}.
\]
In other words \begin{equation}\label{eq:13}
    \lie{p}(P,Q) = \lie{w}(\phi(P) \ostar \phi(Q)) \text{ for all } (P, Q) \in \GTz(\lambda)\times \GTz(\mu)
\end{equation}
where $\phi \times \phi$ is the crystal isomorphism given in the proposition \ref{prop:crysiso}, $\lie p$ is defined in the equation \eqref{eq:per-asso-GT} and $\lie w$ is defined in the equation \eqref{eq:wasso}.
\end{proposition}

\begin{proof}
For a given pair $(P,Q) \in \GTz(\lambda) \times \GTz(\mu)$, by lemma \ref{lem:wasso} we know that
\begin{equation}\label{eq:wasso-LS-p}
    \lie{w}(\phi(P) \ostar \phi(Q)) = (\widehat{\tau}^{-1}*\widetilde{\varphi}w_0)\cdot w_0
\end{equation}
where $\widehat{\tau}$ is the maximal representative of the final direction of $\phi(P)$ and $\widetilde{\varphi}$ is the minimal representative of the initial direction of $\phi(Q)$. The association of Weyl group element to the pair $(P,Q)$ is given by the equation \eqref{eq:per-asso-GT} is 
\begin{equation}\label{eq:wasso-gt-p}
    \lie p (P,Q) = *(w(P,Q))\cdot w_0.
\end{equation}
From the equations \eqref{eq:wasso-LS-p} and \eqref{eq:wasso-gt-p}, it is enough to prove that 
\begin{equation}\label{eq:enough-req}
    \widehat{\tau}^{-1}*\widetilde{\varphi}w_0 = *(w(P,Q)).
\end{equation} 
We have $w(P,Q)  = \lie f (P) \lie i (Q)$ in the procedure \ref{proced} and by lemma \ref{lem:dem} $$*(w(P,Q)) =  *(\lie f (P) \lie i (Q)) = *\lie f (P) * (* \lie i (Q)).$$
Finally by lemma \ref{lem:final} and lemma \ref{lem:init} we have $$ *\lie f (P) * (* \lie i (Q)) = \widehat{\tau}^{-1}*\widetilde{\varphi}w_0 .$$
\end{proof}

 Finally, we can prove the theorem \ref{thm:gt-mod}.
\begin{proof}[Proof of theorem \ref{thm:gt-mod}]
It is clear from the above equation \eqref{eq:13} that
\begin{equation}\label{eq:GT-LS}
 \GTz(\lambda,w,\mu) = \{ (P,Q) \in \GTz(\lambda) \times \GTz(\mu) \ | \ \phi(P) \ostar \phi(Q) \in P(\lambda,w,\mu)  \}
 \end{equation}
 Combining theorem \ref{thm:lspath-mod}, equation \eqref{eq:GT-LS} and $\phi \times \phi : \GTz(\lambda) \times \GTz(\mu) \mapsto P_\lambda \ostar P_\mu$ is a crystal isomorphism proves the theorem \ref{thm:gt-mod}. 
\end{proof}

\subsection{} Next, we will prove lemmas \ref{lem:init} and lemma \ref{lem:final}.

\begin{proof}[Proof of lemma \ref{lem:init}]
Let $P = (a_{ij})_{n\geq i \geq j \geq 1}$ be an integral GT-pattern of shape $\mu$. Define the set $\lie{F}_P:= \{ (i,j) | \ a_{ij}-a_{i-1,j} = 0 \ \text{for} \ n \geq i > j \geq 1 \}$. We associate a word $w(\lie{F}_P)$ corresponding to $\lie{F}_P$ as follows: write $s_{i-j}$ from left to right for all $(i,j) \in \lie{F}_P$ in the total order defined in the section \ref{ss:kogan}. Note that $\lie{i}(P) = w(\lie{F}_P)$. 

From the proposition \ref{prop:fujita}, for any $F \subseteq \lie{F}_P$ such that $\kogan(\mu,F)$ is a reduced Kogan face then $P \in \koganz(\mu,F)$ and hence $P \in \demcrys(\mu, \sigma(F)\wnot)$. Further, if $P \in \demcrys(\mu, w)$ then there exist a subset $F \subseteq \lie{F}_P$ such that $\kogan(\mu,F)$ is reduced, $w=\sigma(F)\wnot$ and $P \in \koganz(\mu,F)$.

There exist a subset $F^0 \subseteq \lie{F}_P$ such that $\kogan(\mu,F^0)$ is reduced, $\sigma(F^0) = *(w(\lie{F}_P))$ and $P \in \demcrys(\mu, \sigma(F^0)\wnot)$. By lemma \ref{lem:maxdeo-str}  $\sigma(F^0) = *(w(\lie{F}_P))$ is the unique maximal element of the set $$\{ \sigma(F) \ | \ F \subseteq \lie{F}_P \text{ and } \kogan(\mu,F) \ \text{is reduced} \}.$$ Hence by above arguments $\sigma(F^0)w_0$ is the unique minimal permutation of the set $$\{ w \ | \ P \in \demcrys(\mu, w)\}.$$ Now we appeal to the proposition \ref{prop:initdem} to conclude that $\sigma(F^0)w_0 = *(w(\lie{F}_P))w_0$ is the minimal representative of the initial direction of $\phi(P)$.
\end{proof}

\begin{proof}[Proof of lemma \ref{lem:final}]
The proof of this lemma is exactly similar to the proof of the above lemma \ref{lem:init}. However, we will record the proof below for completeness.
Le $P = (a_{ij})_{n\geq i > j \geq 1}$ be an integral GT-pattern of shape $\mu$. Define the set $\lie{H}_P:= \{ (i,j) | \ a_{i-1,j}-a_{i,j+1} = 0 \ \text{for} \ n \geq i > j \geq 1 \}$. We associate a word $w(\lie{H}_P)$ corresponding to $\lie{H}_P$ as follows: write $s_{j}$ from left to right for all $(i,j) \in \lie{H}_P$ in the total order defined in the section \ref{ss:dualkogan}. Note that $\text{rev}(\lie{f}(P)) = w(\lie{H}_P)$ where $\text{rev}(\lie{f}(P))$ is the word obtained by writing the word $\lie{f}(P)$ in the reverse order, then $(*(\lie{f}(P)))^{-1} = *(w(\lie{H}_P))$.

From the proposition \ref{prop:fujita1}, for any $F \subseteq \lie{H}_P$ such that $\overline{\kogan}(\mu,F)$ is a reduced dual Kogan face then $P \in \overline{\koganz}(\mu,F)$ and hence $P \in \demcrys(\mu, \overline{\sigma}(F))^{op}$. Further, if $P \in \demcrys(\mu, w)^{op}$ then there exist a subset $F \subseteq \lie{H}_P$ such that $\overline{\kogan}(\mu,F)$ is reduced, $w=\overline{\sigma}(F)$ and $P \in \overline{\koganz}(\mu,F)$.

There exist a subset $\overline{F'} \subseteq \lie{H}_P$ such that $\overline{\kogan}(\mu,\overline{F'})$ is reduced, $\overline{\sigma}(\overline{F'}) = *(w(\lie{H}_P))$ and $P \in \demcrys(\mu, \overline{\sigma}(\overline{F'}))$. By lemma \ref{lem:maxdeo-str}  $\overline{\sigma}(\overline{F'}) = *(w(\lie{H}_P))$ is the unique maximal element of the set $$\{ \overline{\sigma}(F) \ | \ F \subseteq \lie{H}_P \text{ and } \overline{\kogan}(\mu,F) \ \text{is reduced} \}.$$ Hence by above arguments $\overline{\sigma}(\overline{F'})$ is the unique maximal permutation of the set $$\{ w \ | \ P \in \demcrys(\mu, w)^{op}\}.$$ Now we appeal to the proposition \ref{prop:finaldem} to conclude that $\overline{\sigma}(\overline{F'}) = *(w(\lie{H}_P))$ is the maximal representative of the final direction of $\phi(P)$.
\end{proof}

\section{Data Availibility Statement}

This manuscript has no associated data.

\bibliographystyle{plain}

\begin{thebibliography}{10}
\expandafter\ifx\csname url\endcsname\relax
  \def\url#1{{\tt #1}}\fi
\expandafter\ifx\csname urlprefix\endcsname\relax\def\urlprefix{URL }\fi

\bibitem{deo:ca:87} 
Deodhar, Vinay V.,
   {\em A splitting criterion for the {B}ruhat orderings on {C}oxeter groups}, Communications in Algebra, 15, (1987), pp. 1889--1894,
      ISSN : 0092-7872,
  \urlprefix\url{https://doi.org/10.1080/00927878708823511}.

\bibitem{fujita}
   {Fujita, Naoki},
   {\em Schubert calculus from polyhedral parametrizations of Demazure crystals},
{Advances in Mathematics},
{397, \text{paper number. }108201},
{2022},
\url{https://doi.org/10.1016/j.aim.2022.108201}.


\bibitem{GT}
    Gelfand, I. M. and Tsetlin, M. L.,
    {\em Finite-dimensional representations of the group of unimodular matrices},
    Dokl. Akad. Nauk SSSR,
    71,
    pages:825--828,
    (1950)
    
\bibitem{kashiwara}
  {Kashiwara, Masaki},
  {\em The crystal base and Littelmann’s refined Demazure character formula},
  {Duke Mathematical Journal},
  {71},
  pages:{839--858},
  {(1993)}


\bibitem{kogan}
{Mikhail Kogan},
{\em Schubert Geometry of Flag Varieties and Gelfand-Cetlin Theory},
   {MIT},
 {(2000)}


\bibitem{Kumar}
  {Kumar, Shrawan},
     {\em Proof of the {P}arthasarathy-{R}anga {R}ao-{V}aradarajan
              conjecture},
    {Inventiones Mathematicae},
    {93},
     {(1988)},
    PAGES:{117--130}


\bibitem{kumar:refineprv}
   {Kumar, Shrawan},
     {\em A refinement of the {PRV} conjecture},
 {Inventiones Mathematicae},
   {97},
     {(1989)},
     PAGES:{305--311},
    \url{https://doi.org/10.1007/BF01389044}.


\bibitem{KRV}
M S Kushwaha and K N Raghavan and S Viswanath,
{\em A study of Kostant-Kumar modules via Littelmann paths},
Advances in Mathematics,
{\bf 381},
{107614},
{(2021)}



\bibitem{KRV_Sat}
{M S Kushwaha and K N Raghavan and S Viswanath},
{\em The saturation problem for refined Littlewood-Richardson coefficients},
{Séminaire Lotharingien de Combinatoire},
{\bf 85B.52},
{12 pp},
{(2021)}


\bibitem{litt:inv}
  {Littelmann, Peter},
  {\em A {L}ittlewood-{R}ichardson rule for symmetrizable {K}ac-{M}oody algebras},
 {Inventiones Mathematicae},
 {\bf 116},
{(1994)},
     PAGES:{329--346},
\url{https://doi.org/10.1007/BF01231564}.



\bibitem{LittelmannAnnals}
{Littelmann, Peter},
  {\em Paths and root operators in representation theory},
{Annals of Mathematics. Second Series},
 {\bf 142},
 {(1995)},
     PAGES:{499--525}


\bibitem{mathieu}
 {Mathieu, Olivier},
  {\em Construction du groupe de {K}ac-{M}oody et applications},
{Comptes Rendus des S\'{e}ances de l'Acad\'{e}mie des Sciences. S\'{e}rie I. Math\'{e}matique},
 {\bf 306},
  {(1988)},
  PAGES:{227--230}.
  
\end{thebibliography}

\end{document}